\documentclass{amsart}
\usepackage{amsmath,amsfonts,amssymb,amsthm}
\usepackage{mathrsfs}
\usepackage{graphicx}

\newtheorem{theorem}{Theorem}[section]
\newtheorem{lemma}[theorem]{Lemma}
\newtheorem{proposition}[theorem]{Proposition}
\newtheorem{corollary}[theorem]{Corollary}
\theoremstyle{definition}

\newtheorem{question}[theorem]{Question}
\theoremstyle{remark}
\newtheorem{remark}[theorem]{Remark}

\DeclareMathOperator{\tr}{tr}
\DeclareMathOperator{\PSL}{PSL}
\DeclareMathOperator{\SL}{SL}
\newcommand{\Cset}{\mathbb{C}}
\newcommand{\Rset}{\mathbb{R}}
\newcommand{\Qset}{\mathbb{Q}}
\newcommand{\Zset}{\mathbb{Z}}
\newcommand{\Hset}{\mathbb{H}}

\title{The realization problem for J\o rgensen numbers}

\author{Yasushi Yamashita}
\address{Nara Women's University, Kitauoyanishi-machi, Nara-shi, Nara 630-8506, Japan}
\email{yamasita@ics.nara-wu.ac.jp}
\thanks{This work was supported by JSPS KAKENHI Grant Number 26400088.}

\author{Ryosuke Yamazaki}
\address{Gakushuin Boys' Senior High School, 1-5-1 Mejiro, Toshima-ku, Tokyo 171-8588 Japan}
\email{rsk.yamazaki.ms@gmail.com}



\begin{document}

\begin{abstract}
Let $G$ be a two generator subgroup of $\PSL(2, \Cset)$.
The J\o rgensen number $J(G)$ of $G$ is defined by
\[ 
    J(G) = \inf\{ |\tr^2 A-4| + |\tr[A,B]-2| \: ; \:
    G=\langle A, B\rangle \}.
\]
If $G$ is a non-elementary Kleinian group,
then $J(G)\geq 1$.
This inequality is called J\o rgensen's inequality.
In this paper, we show that,
for any $r\geq 1$, there exists a non-elementary 
Kleinian group whose J\o rgensen number is equal to $r$.
This answers a question posed by Oichi and Sato.
We also present our computer generated picture
which estimates J\o rgensen numbers from above
in the diagonal slice of Schottky space.
\end{abstract}

\maketitle

\section{Introduction}

Let $G$ be a two generator subgroup of $\PSL(2, \Cset)$.
Determining whether or not $G$ is discrete is 
an important problem in Kleinian group theory.
In 1976, J\o rgensen \cite{MR0427627} showed that 
if $G=\langle X, Y\rangle$ is a rank two non-elementary discrete
subgroup of $\PSL(2, \Cset)$ (Kleinian group), then
\begin{equation}\label{eqn:Jorgensenineq}
  |\tr^2 X - 4| + |\tr [X, Y] -2| \geq 1,
\end{equation}
where $[X, Y] = X Y X^{-1} Y^{-1}$.
The J\o rgensen number of an ordered pair $(X, Y)$ is defined as 
\[
  J(X, Y) := |\tr^2 X - 4| + |\tr [X, Y] -2|
\]
and the J\o rgensen number of a rank two non-elementary Kleinian group $G$
is defined as
 \[
  J(G) := \inf \{ J(X, Y) | \langle X, Y\rangle = G \}.
\]
By (\ref{eqn:Jorgensenineq}), we have $J(G)\geq 1$.
We can think of J\o rgensen number as measuring
how far from being indiscrete.
J\o rgensen's inequality is sharp, and
if $J(G)=1$, $G$ is called a J\o rgensen group,
There are many results on J\o rgensen groups 
in the literature
\cite{MR1040927,MR0399452,MR2107659,MR2205423,MR2153913,MR2107132,MR1759686,MR2032679,MR3289108}.
Also, calculating J\o rgensen numbers can be 
difficult and interesting challenge
\cite{MR2525101,MR2351012,MR2199365}.
Oichi-Sato \cite{OS} asked the following natural question:

\begin{question}[The realization problem]
Let $r$ be a real number with $r\geq 1$.
When is there a non-elementary Kleinian group
whose Jorgensen number is equal to $r$?
\end{question}

Oichi-Sato \cite{OS} claimed that if $r=1, 2, 3$ or $r\geq 4$, 
then there is a non-elementary Kleinian group 
whose J\o rgensen number is equal to $r$.
(See also \cite{MR2525101,MR3289108,YR}.)
Though J\o rgensen's inequality is sharp,
constructing a Kleinian group with small J\o rgensen number
(in particular, less than $4$) is not easy.
To the best of our knowledge,
there is no known example of a classical Schottky group
whose J\o rgensen number is less than $4$.
In this paper, we solve this realization problem.

\begin{theorem}\label{thm:main}
For any real number $r\geq 1$,
there is a rank two non-elementary discrete subgroup of $\PSL(2,\Cset)$
whose J\o rgensen number is $r$.
\end{theorem}

The J\o rgensen numbers can be studied by computer
experiments.
We will present a computer generated picture
of the estimates of J\o rgensen numbers from above
in the diagonal slice of Schottky space.
At the end of the paper,
we present some J\o rgensen groups,
and we conjecture that they are
counter examples of Li-Oichi-Sato's conjecture.

In section 2, we give a proof of Theorem \ref{thm:main} for the case $1 \leq r \leq 4$.
In section 3, we give a proof for the case $4 \leq r$.
In section 4, we present our computer generated picture,
and describe our calculation.

\section{The proof for the case $1 \leq r \leq 4$}

In this section, we give a proof for the case $1 \leq r \leq 4$.
For simplicity,
we use the same notation for a matrix in $\SL(2, \Cset)$
and its equivalence class in $\PSL(2, \Cset)$.

\subsection{The basic configuration}

Let $z, z'\in \Cset$,
and $[z, z']$ be the oriented line from $z$ to $z'$ in $\Hset^3$.
Following \cite{MR1004006}, we define the {\em line matrix} associated with $[z, z']$ as
\[
  M([z, z']) := \frac{i}{z'-z} \begin{pmatrix} z+z'  & -2zz' \\ 2 & -z-z' \end{pmatrix}.
\]
See \cite{MR1004006}, p. 64, equation (1).
It is the order two rotation about $[z, z']$.

For $a\geq 1$, set
\[
  P_a := M([a, -3a]), \quad Q := M([1, -1]), \quad R := M([0, \infty]).
\]
Let $G_a$ be a subgroup of $\PSL(2, \Cset)$ generated by $P_a, Q$ and $R$.
The axes of the generators are contained 
in the vertical plane $ \{(x,0,z) | z>0\} \cong \Hset^2$.
Since we assumed that $a\geq 1$, $[a, -3a]$ and $[1, -1]$ do not intersect in $\Hset^3$.
The angle between $[1,-1]$ and $[0, \infty]$ is $\pi/2$, and
the angle between $[0, \infty]$ and $[a, -3a]$ is $\pi/3$.
\begin{figure}
\centering
\includegraphics[height=4cm]{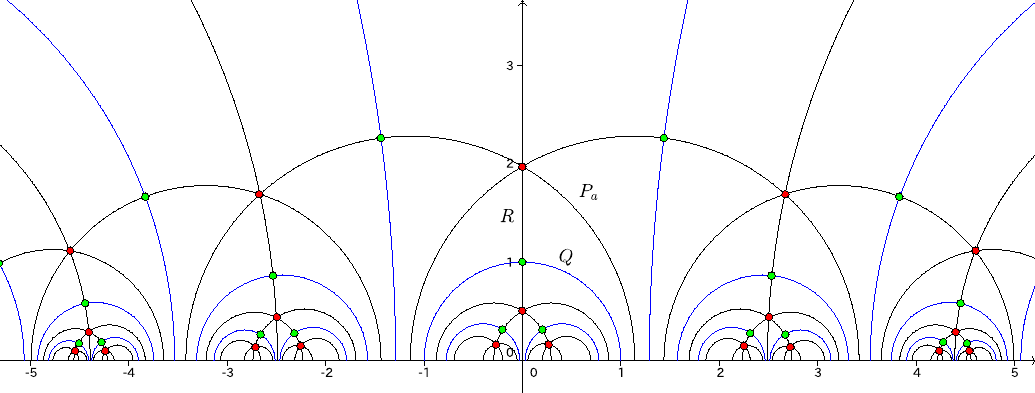}
\caption{The axis of the generators and their orbits.}
\label{fig:axis}
\end{figure}
See Figure~\ref{fig:axis}.
Hence, as an abstract group, $G_a$ has a presentation
\[
  \left\langle P_a, Q, R \:|\: P_a^2 = Q^2 = R^2 = (QR)^2 = (R P_a)^3 = id \right\rangle
\]
for any $a\geq 1$ and is isomorphic to a hyperbolic (full) triangle group in $\Hset^2$.
(Note that hyperbolic (full) triangle groups are 
defined as groups generated by reflections,
but $G_a$ is generated by rotations in $\Hset^3$.)
Hence, $G_a$ is non-elementary and discrete.
For later purpose, we classify non-trivial elements of $G_a$ into three types:
\begin{enumerate}
    \item[(i)] elliptic elements
    \item[(ii)] loxodromic elements which can be conjugated into 
    $\langle P_a, Q\rangle$
    \item[(iii)] loxodromic elements which can not be conjugated into
    $\langle P_a, Q\rangle$
\end{enumerate}
When $a=1$, we define that the parabolic elements are of type (ii).

\begin{remark}
Groups similar to
$\langle P_a, Q, R\rangle$ were studied 
by C.~Series, S.~P.~Tan and the first author
in \cite{STY}.
\end{remark}

\begin{lemma}\label{lem:trace}
Let $E_a$ be an element of type (iii) in $G_a$.
Then, we have
\[
    \tr E_a \in \{ \pm x \:|\: x\in \Rset, x \geq 3\} \cup 
    \{ \pm y i \:|\: y\in \Rset, y \geq 1 \}.
\]
\end{lemma}

\begin{proof}
We begin by considering the case $a=1$.
Set
\[
    M = \begin{pmatrix} \frac{1}{\sqrt{2}} & \frac{1}{\sqrt{2}} \\
                        0 & \sqrt{2}
                        \end{pmatrix}.
\]
Then, we have
\[
    M P_1 M^{-1} = \begin{pmatrix}  0 & -i\\  -i &  0\end{pmatrix},
    M Q   M^{-1} = \begin{pmatrix} -i &  0\\ -2i &  i\end{pmatrix},
    M R   M^{-1} = \begin{pmatrix}  i & -i\\   0 & -i\end{pmatrix}.
\]
Hence, 
for any element $N\in G_1$,
$\tr N = \pm n$ or $\pm n \: i$ for some integer $n$.
Let $E_1$ be an element of type (iii) in $G_1$.
Since $E_1$ is loxodromic, 
\begin{equation}\label{eqn:e1}
\tr E_1 \in 
\{\pm 3, \pm 4, \pm 5, \ldots \} \cup \{ \pm i, \pm 2 i, \pm 3 i\ldots \}.
\end{equation}

Now, we consider the general case.
For $a > 1$, 
let $E_a$ be a type (iii) element of $G_a$.
Since $E_a$ is loxodromic and preserves the 
vertical plane $\{(x,0,z) | x>0\} \cong \Hset^2$,
$E_a$ or $E_a^{-1}$ is conjugate to 
\[
    \begin{pmatrix} \sqrt{l_a}  & 0 \\ 0 & 1/\sqrt{l_a}  \end{pmatrix} \text{ or }
    \begin{pmatrix} \sqrt{l_a}i & 0 \\ 0 & 1/\sqrt{l_a}i \end{pmatrix},
\]
where $l_a$ is the translation length of $E_a$.
Thus, we have
\[
    \tr E_a = \pm (\sqrt{l_a} + 1/\sqrt{l_a}) \text{ or } 
    \pm (\sqrt{l_a} - 1/\sqrt{l_a})i.
\]
Note that the region $F_a$ in the vertical plane $\Hset^2$
bounded by the axes of $P_a$, $Q$, $R$ is 
a fundamental region for $G_a$ in $\Hset^2$.
Since $F_a$ is monotonically increasing in $a$,
$l_a$ is also monotonically increasing in $a$.
Combining (\ref{eqn:e1}), we see that
\[
    \tr E_a \in \{ \pm x \:|\: x\in \Rset, x \geq 3\} \cup 
    \{ \pm y i \:|\: y\in \Rset, y \geq 1 \},
\]
and the lemma is proved.
\end{proof}

\begin{remark}
By this proof, we see that
$G_1$ is a conjugate of a subgroup of Picard group.
Sato
\cite{MR2199365}
showed that Picard group is a Jorgensen group.
Gonz{\'a}lez-Acu{\~n}a and Ram{\'{\i}}rez
\cite{MR2351012}
described all J\o rgensen subgroups of Picard group.
\end{remark}

\subsection{The singular solid torus}\label{subsec:tsst}

We define
\[
A_a := P_a Q = \begin{pmatrix} -3a/2 & 1/2\\ -1/2 & -1/2a \end{pmatrix},  \quad 
B  := R = \begin{pmatrix} i & 0 \\ 0 & -i \end{pmatrix}.
\]
Since $Q = A_a B A_a^{-1} B^{-1} A_a B$, 
the rank of $G_a$ is two and
$(A_a, B)$ is a generating pair of $G_a$.
As an abstract group, 
by Tietze transformation, we have
\begin{align}
  G_a & = \langle P_a, Q, B, A_a \: | \: P_a^2 = Q^2 = B^2 = (QB)^2 = (B P_a)^3 = id, A_a = P_a Q \rangle \nonumber \\ 
  & = \left\langle A_a, B \: \Big| \: 
  \begin{array}{l} 
    (A_a B A_a^{-1} B A_a B A_a^{-1})^2 = (A_a B A_a^{-1} B^{-1} A_a B)^2 = B^2 \nonumber \\ 
    = (A_a B A_a^{-1} B^{-1} A_a B \cdot B)^2 
    = (B \cdot A_a B A_a^{-1} B A_a B A_a^{-1})^3 = id
  \end{array}
  \right\rangle \\
  & =  \langle A_a, B \: | \:  (A_a B A_a^{-1} B^{-1} A_a B)^2 = B^2
  = (A_a B A_a^{-1} B^{-1} A_a)^2 =id \rangle. \label{eqn:group}
\end{align}

\begin{remark}
Generating pairs similar to $A_a$ and $B$ were studied in \cite{STY},
and called the {\it singular solid torus}.
\end{remark}

\begin{lemma}\label{lem:orderthree}
In $G_a$, any element of order three 
can not be a part of minimal generating system.
\end{lemma}

\begin{proof}
Let $\tau$ be the natural projection from the free group 
$F_2 = \langle A_a, B\rangle$ of rank two to
$\langle A_a\: | \: A_a^2\rangle \oplus \langle B\:|\: B^2\rangle \cong 
\{(i, j)\:|\: i, j \in \{0, 1\}\}$.
Since
\[
    \tau( (A_a B A_a^{-1} B^{-1} A_a B)^2 ) = 
    \tau( B^2 ) = 
    \tau( (A_a B A_a^{-1} B^{-1} A_a)^2 ) = (0, 0),
\]
$\tau$ is well-defined on $G_a$.
We denote the set $\{ X\in G_a\:|\: \tau(X)=(i, j)\}$ by $G_{(i,j)}$.
$G_{(i,j)}$ is not empty set for $(i,j) = (0,0), (0,1), (1,0), (1,1)$.

Let $X$ be an element of $G$ such that $X^3=id$.
Then, $\tau(X)=(0, 0)$.
(Otherwise, we have 
$\tau(X^3) \equiv \tau(X^2) + \tau(X) \equiv \tau(X) 
\not\equiv (0, 0) = \tau(id)$ (mod 2),
which is a contradiction.)
Hence, for any $Y\in G$, 
we have
\[
    \langle X, Y\rangle \subset G_{(0,0)}\cup G_{\tau(Y)} \neq 
    G_{(0,0)} \cup G_{(1,0)} \cup G_{(0,1)} \cup G_{(1,1)} =
    G
\]
and the lemma is proved.
\end{proof}

\subsection{The J\o rgensen number}\label{subsec:JN}
Finally, we calculate the J\o rgensen number. 



\begin{proposition}
For any $1 \leq a \leq a_0$, where $a_0 = (\sqrt{7}+2)/3$, we have
\[
    J(G_a) = J(A_a, B) = \frac{(3a^2 - 1)^2}{4a^2}.
\]
In particular, $J(G_1) = 1$ and $J(G_{a_0}) = 4$.
For any $1\leq r \leq 4$, the J\o rgensen number is realized.
\end{proposition}

\begin{proof}
Since
$\tr^2 A = (3a^2+1)^2/4a^2$ and
$\tr[A_a, B] = 1$, for $a\geq 1$, we have
\[
    J(A_a, B) = 
    \left| \frac{(3a^2+1)^2}{4a^2} - 4\right| + 1 =
    \frac{(3a^2-1)^2}{4a^2}.
\]
Since $a_0$ satisfies the equation $(3a_0^2-1)^2/4a_0^2 = 4$,
if $1 \leq a \leq a_0$, we have $1 \leq J(A_a, B) \leq 4$.
To prove this proposition by contradiction, 
suppose that there exist $C, D\in G_a$ such that $J(C, D) < J(A_a, B)$
and $G_a=\langle C, D\rangle$. Since $G_a$ is non-elementary, $|\tr[C,D]-2|>0$. (See p.68 in \cite{MR698777}.) 

\noindent{\bf Case 1}
$C$ is of type (i).
If $C$ is an elliptic element of order two,
then $\tr C = 0$ and 
$J(C,D) = |0^2 - 4| + |\tr[C,D]-2| > 4 \geq J(A_a, B)$,
and this can not happen.
If $C$ is an elliptic element of order three,
since $C$ can not be a part of minimal generating system
by Lemma \ref{lem:orderthree},
this can not happen.

\noindent{\bf Case 2}
$C$ is of type (ii).
Let $C'$ be an element of $\langle P_a, Q\rangle$
which is conjugate to $C$.
Since $\langle P_a, Q\rangle$ is isomorphic to the infinite dihedral group,
$C'$ can be written as 
$C' = P_a Q P_a Q \cdots$ or 
$C' = Q P_a Q P_a \cdots$.
If the word length is odd,
then $C'$ is conjugate to $P_a$ or $Q$ and $C'$ is not loxodromic.
Thus, the word length is even,
and $C' = (P_a Q)^n$ or $C' = (Q P_a)^{-n} = (P_a Q)^n$ 
for some $n \in \Zset$.

\noindent{\bf Case 2-1} 
$|n|>1$.  
We claim that $G_a \neq \langle C, D \rangle$.
To see this, 
consider the  Coxeter group
 \[
  H_n := \langle P, Q, R \:|\: 
  P^2 = Q^2 = R^2 = (PR)^3 = (QR)^2 = (PQ)^n = id \rangle
\]
and let $\pi: G_a \to H_n$ be the natural surjection.
Since $\pi(C)=id$, we have 
\[
\pi(\langle C, D\rangle) = \langle \pi(D) \rangle \neq H_n,
\]
because $H_n$ is not cyclic when $|n|>1$.
It follows that $G_a \neq \langle C, D\rangle$ and the claim is proved.

\noindent{\bf Case 2-2}
$|n| = 1$.  In this case, $\tr^2 C = \tr^2 A_a$.

\noindent{\bf Case 2-2-1} $[C, D]$ is of type (i).
If $[C, D]$ is elliptic of order two,
then 
\[
    J(C,D) = 
    |\tr^2 A_a - 4| + |0 - 2| 
    > |\tr^2 A_a - 4| + |1 - 2| 
    = J(A_a, B).
\]
If $[C, D]$ is elliptic of order three,
then $\tr[C, D] = \pm 1$, and
\[
    J(C,D) = |\tr^2 C - 4| + |\tr[C,D] - 2| = |\tr^2 A_a - 4| + |\pm 1 - 2| \geq J(A_a, B).
\]

\noindent{\bf Case 2-2-2} $[C, D]$ is of type (ii).
Let $E_a'$ be an element of $\langle P_a, Q\rangle$
which is conjugate to $[C, D]$.
Then, $E_a' = A_a^m$ for some $m \neq 0$.
(Recall the first paragraph of Case 2.)
We have $\lim_{a\to 1} \tr E_a' = \lim_{a\to a} \tr A_a^m = 2$ or $-2$.
For $1 \leq b \leq a$, 
we denote the elements in $G_b$
which corresponds to $C$ and $D$
by $C_b$ and $D_b$.

If $\lim_{a\to 1} \tr E_a' = 2$, then
$J(C_b, D_b)$ is close to $0$ when $b$ is close to $1$,
which contradicts 
J\o rgensen's inequality,
because $G_b$ is non-elementary and discrete.

If $\lim_{a\to 1} \tr E_a' = -2$, then $\tr [C, D] \leq -2$ and 
\[
    J(C,D) 
    \geq |\tr^2 A - 4| + | -2 -2| > J(A_a, B).
\]

\noindent{\bf Case 2-2-3} $[C, D]$ is of type (iii).
By Lemma~\ref{lem:trace}, we have
\[
    J(C, D) = |\tr^2 C - 4| + |\tr [C, D] - 2| \geq 
    |\tr^2 A_a - 4 | + 1 = J(A_a, B).
\]

\noindent{\bf Case 3}
$C$ is of type (iii).
By lemma~\ref{lem:trace}, we have
\[
    J(C, D) = |\tr^2 C -4| + |\tr[C, D] - 2|
    \geq |\tr^2 C -4| \geq 5 \geq J(A_a, B).
\]

Hence, $J(G_a) = J(A_a, B)$ for $1\leq a \leq a_0$,
and the proposition is proved.
\end{proof}

\section{The proof for the case $r\geq 4$}\label{sec:Schottky}

Oichi and Sato \cite{OS} claimed that, 
for every real number $r \geq 4$,
there is a subgroup $G$ of $\PSL(2, \Cset)$
such that $J(G)=r$.
But the proof was not written.
For completeness,
we give a proof of this fact
by calculating the J\o rgensen number 
of some Kleinian groups.

\subsection{Markoff maps}

First, we recall Bowditch, Tan-Wong-Zhang theory
\cite{MR1643429, MR2370281}
on Markoff maps very briefly.
See section 3 in \cite{MR2370281} for detail.

We denote the set $\Qset \cup \{1/0\}$ by $\widehat\Qset$.
Let $\mathcal F$ be the Farey triangulation of 
the upper half plane $\Hset^2$.
Recall that the vertex set of $\mathcal F$ is 
$\widehat\Qset$, 
and two vertices $p/q$ and $r/s$ are connected by a geodesic
in $\Hset^2$ if $ps-qr=\pm 1$.
Let $\Sigma$ be the binary tree dual to $\mathcal F$.
See Figure \ref{fig:Farey}.
\begin{figure}
\centering
\includegraphics[height=5cm]{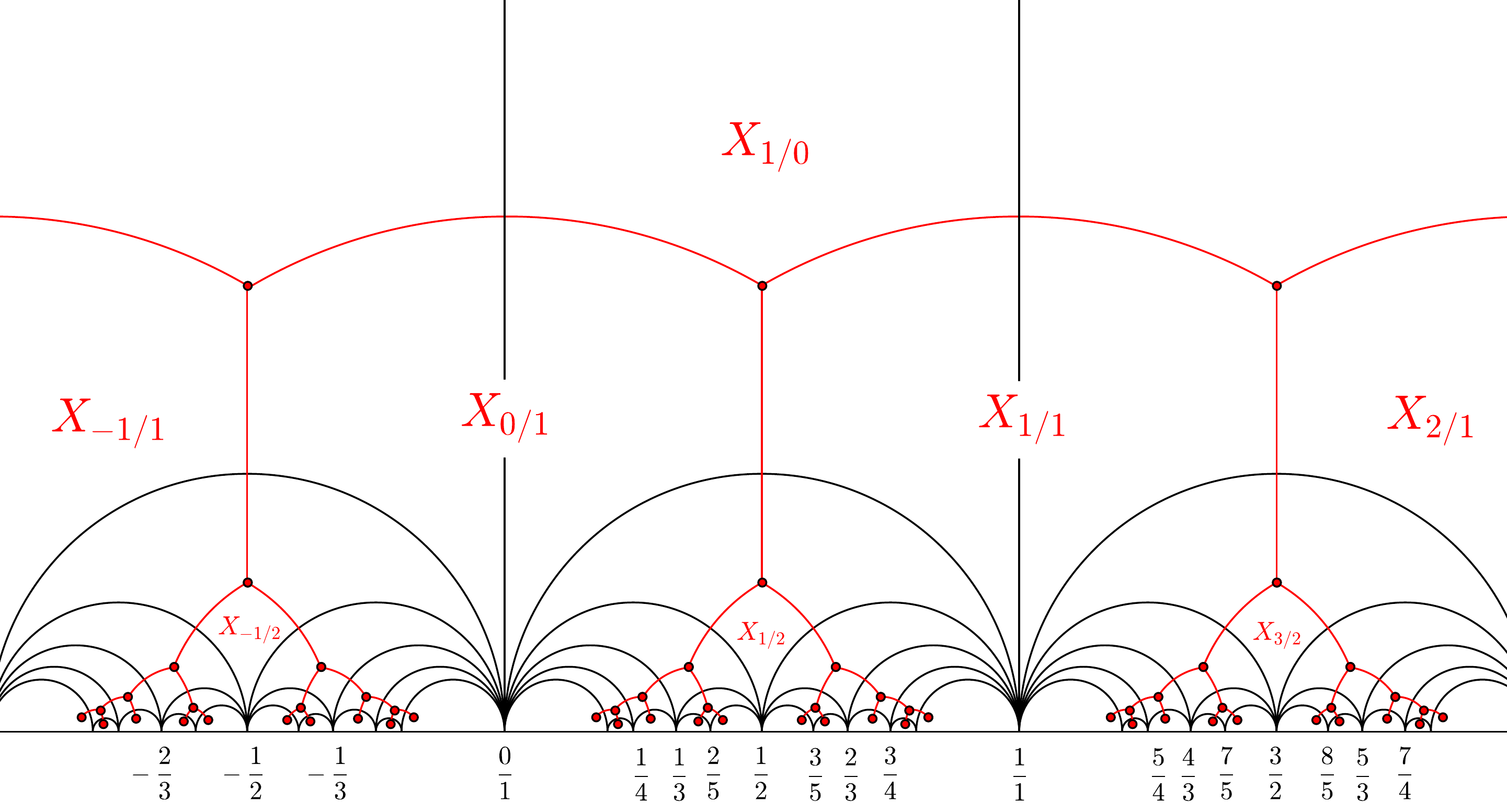}
\caption{Farey triangulation and the dual graph $\Sigma$ (red).}
\label{fig:Farey}
\end{figure}
A {\it complementary region} of $\Sigma$ is the closure of 
a connected component of the complement of $\Sigma$.
The set of complementary regions of $\Sigma$ is denoted by $\Omega$.
Since each complementary region corresponds 
to a vertex in $\mathcal F$,
we can identify $\Omega$ with $\widehat\Qset$,
and we denote the complementary region which corresponds to
$p/q\in\Qset$ by $X_{p/q}$.
Let $e$ be an edge of $\Sigma$ with end points $u$ and $v$.
Then, there exist $X, Y, U, V\in\Omega$ such that
$e = X\cap Y$, $u = X\cap Y\cap U$ and $v = X\cap Y\cap V$.
See Figure ~\ref{fig:edge}.
\begin{figure}
\centering
\includegraphics[height=1.5cm]{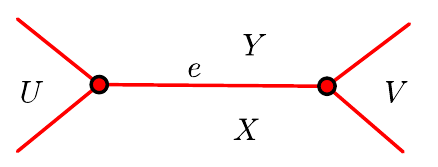}
\caption{$e = (X, Y; U, V)$}
\label{fig:edge}
\end{figure}
We write $e = (X, Y; U, V)$ to indicate these regions.
A {\it Markoff map} is a map $\psi$ from $\Omega$ to $\Cset$
such that, for every edge  $e=(X, Y; U, V)$ in $\Sigma$,
we have
\[
    \psi(U) + \psi(V) = \psi(X)\psi(Y).
\]
This condition is called the {\it edge relation}.
Given $k\geq 0$, the set $\Omega_\psi(k)$ is defined by
$\Omega_\psi(k) = \{ X\in\Omega \:|\: |\psi(X)| \leq k \}$.
We will need the following lemma.
\begin{lemma}[Theorem 3.1 (2), \cite{MR2370281}]\label{lem:connected}
Let $\psi$ be a Markoff map.
For any $k\geq 2$, 
the union $\cup_{X\in\Omega_\psi(k)} X$
is connected as a subset of $\Hset^2$.
\end{lemma}

Let $F_2 = \langle A, B\rangle$ be the free group on $A$ and $B$.
An element $W$ of $F_2$ is called {\it primitive}
if there exists an element $V$ such that 
$F_2 = \langle W, V\rangle$,
and $W$ and $V$ 
are called {\it associated primitives}.
Let $\alpha$ be the abelianization homomorphism from
$F_2$ onto $F_2/F_2'$ ($\cong \Zset^2$).
Let $\rho$ be an $\SL(2, \Cset)$ representation of $F_2$.
Then, the next map $\psi_\rho : \Omega \to \Cset$
is well-defined:
\[
    \psi_\rho(X_{n/m}) = \tr\rho(W_{n/m}),
\]
where $W_{n/m}$ is a primitive element such that 
$\alpha(W_{n/m}) = A^m B^n$.
See Corollary 3.2, \cite{MR608526}.
For example, we have
\[
    \psi_\rho(X_{0/1}) = \tr_\rho(A), \quad
    \psi_\rho(X_{1/0}) = \tr_\rho(B), \quad
    \psi_\rho(X_{1/1}) = \tr_\rho(AB).
\]
By Theorem 1.2 and 1.3 in \cite{MR608526} 
and the trace identity in $\SL(2, \Cset)$
\begin{equation} \label{eqn:trace}
    \tr AB + \tr AB^{-1} = \tr A \tr B,
\end{equation}
we have the next lemma.
(See Section 3, 
Natural correspondence ${\mathcal X}_{\mu-2} = \psi_\mu$ 
\cite{MR2370281}.)
\begin{lemma}\label{lem:markoff}
$\psi_\rho$ is a Markoff map.
\end{lemma}

For any associated primitives $\{A, B\}$ and $\{C, D\}$ of $F_2$,
the commutator $[C, D]$ is conjugate to $[A, B]$ or $[A, B]^{-1}$.
(See Theorem 3.9 in \cite{MR0422434}.)
Thus, we have $\tr\rho[A, B] = \tr\rho[C, D]$.
It follows that

\begin{lemma}\label{lem:commutator}
If $G$ is a subgroup of $\PSL(2, \Cset)$ isomorphic to $F_2$,
then the trace of the commutator of associated primitives
of $G$ does not depend on the choice of associated primitives.
\end{lemma}

Recall that,
if $G$ is in the {\it Maskit slice} \cite{MR1241870,MR1913879},
then it is discrete and isomorphic to $F_2$ and
can be normalized as $G_\mu = \langle A, B_\mu \rangle$,
where
\[
    A = \begin{pmatrix} 1 & 2 \\ 0 & 1 \end{pmatrix}, \;
    B_\mu = \begin{pmatrix} -i\mu & -i \\ -i & 0\end{pmatrix}
    \text{ for some } \mu\in\Cset.
\]

\begin{corollary}
If $G_\mu$ is in the {\it Maskit slice},
then $J(G_\mu) = 4$.
\end{corollary}

\begin{proof}
Since $|\tr^2 A - 4| = 0$,
by Lemma~\ref{lem:commutator}, 
we have $J(G_\mu) = J(A, B_\mu) = 4$..
\end{proof}

\subsection{Kissing Schottky groups}

We consider {\it kissing Schottky groups} studied in \cite{MR1913879}.
For a positive real number $k$, 
let $\rho_k$ be the $\SL(2, \Cset)$ representation 
of $F_2 = \langle A, B \rangle$ given by
\[
    \rho_k(A) = \begin{pmatrix} x & ik\cdot y \\ y/(ik) & x \end{pmatrix}, \quad
    \rho_k(B) = \begin{pmatrix} x & y \\ y & x \end{pmatrix},
\]
where $x, y$ are positive real numbers 
with $x^2 = y^2 + 1$ and $y^2 = 2/(k+1/k)$.
(See chapter 6, p.170 \cite{MR1913879}.)
The first condition guarantees that 
$A$ and $B$ have determinant $1$.
The second condition guarantees that
$\tr[A, B] = -2$.
See Figure 6.8 in \cite{MR1913879}.
For the sake of simplicity, 
we denote the Markoff map $\psi_{\rho_k}$ 
of the representation $\rho_k$ by $\psi$.

\begin{lemma}\label{lem:min}
We have $\Omega_\psi(2x) = \{ X_{0/1}, X_{1/0} \}$.
\end{lemma}

\begin{proof}
Since $\psi(X_{0/1}) = \tr\rho_k(A) = 2x$ and
$\psi(X_{1/0}) = \tr\rho_k(B) = 2x$, 
by definition,
we have
$\{X_{0/1}, X_{1/0}\} \subset \Omega_\psi(2x)$.

Complementary regions which meet $X_{1/0}$ are
\[
    \ldots, X_{-3/1}, X_{-2/1}, X_{-1/1}, X_{0/1}, 
    X_{1/1}, X_{2/1}, X_{3/1}, \ldots.
\]
See Figure \ref{fig:Farey}.
Put $x_n = \psi(X_{n/1})$ for $n\in\Zset$.
Then, we have
\begin{align*}
    x_{-1} & = \psi(X_{-1/1}) = \tr\rho_k(AB^{-1}) =
    \frac{2(k+1)^2}{k^2+1} - \frac{2k^2-2}{k^2+1}i,\\
    x_0 & = \psi(X_{0/1}) = \tr\rho_k(A) =
    \frac{2(k+1)}{\sqrt{k^2+1}},\\
    x_1 & = \psi(X_{1/1}) = \tr\rho_k(AB) = 
    \frac{2(k+1)^2}{k^2+1} + \frac{2k^2-2}{k^2+1}i.
\end{align*}
By the edge relation, we have 
\[
    x_{n+1} = 2x \cdot x_n - x_{n-1}
    = \frac{2(k+1)}{\sqrt{k^2+1}} \cdot x_n - x_{n-1}
\]
for any $n\in\Zset$.
Observe that $|x_{n+1}| > |x_n|$ for any $n\geq 0$.
First, direct calculation shows that $|x_1| > |x_0|$.
Then, since $2(k+1)/\sqrt{k^2+1} > 2$, the edge relation 
implies that $|x_{n+1}| > |x_n|$.
We also see that 
$|x_{n-1}| > |x_n|$ for any $n\leq 0$.
(In fact, $|x_n| = |x_{-n}|$.)
Hence, $X_{0/1}$ is the only region in $\Omega_\psi(2x)$
which meets $1/0$.

Complementary regions which meet $X_{0/1}$ are
\[
    \ldots, X_{-1/3}, X_{-1/2}, X_{-1/1}, X_{1/0},
    X_{1/1}, X_{1/2}, X_{1/3}, \ldots .
\]
By the edge relation, we have $\psi(X_{1/n}) = \psi(X_{n/1})$.
Hence, $X_{1/0}$ is the only region in $\Omega_\psi(2x)$
which meets $X_{0/1}$,

Since $2x = 2(k+1)/\sqrt{k^2+1}  > 2$, 
by Lemma \ref{lem:connected}, 
$\Omega_\psi(2x)$ is connected.
Hence, 
$\Omega_\psi(2x) = \{X_{0/1}, X_{1/0}\}$,
and the lemma is proved.
\end{proof}

By this lemma, 
we have $|\tr\rho_k(C)| \geq |\tr\rho_k(A)|$
for any primitive element $C\in F_2$.
Since $\tr\rho_k(A)$ is real and greater than $2$,
we have $|\tr^2\rho_k(C) - 4| \geq |\tr^2\rho_k(A) - 4|$.
Since $\rho_k(F_2)$ is isomorphic to $F_2$, 
by Lemma \ref{lem:commutator},
we have the following.
\begin{proposition}
For any kissing Schottky group (representation) $\rho_k$, 
we have
\[
    J(\rho_k(F_2)) 
    = J(\rho_k(A), \rho_k(B))
    = \frac{4(k+1)^2}{k^2+1}.
\]
In particular,
we have $J(\rho_1(F_2)) = 8$, and
$\lim_{k\to\infty} J(\rho_k(F_2)) = 4$,
For any real $r$ with $4 < r \leq 8$, 
the J\o rgensen number is realized
by a kissing Schottky group.
\end{proposition}


\subsection{$\theta$-Schottky groups}

Here, we consider Fuchsian Schottky groups
described in Project 4.2, p.118 in \cite{MR1913879}.
For $0 < \theta \leq \pi/4$, 
let $\rho_\theta$ be the $\SL(2, \Cset)$ representation
of $F_2 = \langle A, B\rangle$ given by
\[
    \rho_\theta(A) = \frac 1 {\sin\theta} 
    \begin{pmatrix} 1 & i\cos\theta \\ 
    -i\cos\theta & 1 \end{pmatrix}, \quad
    \rho_\theta(B) = \frac 1 {\sin\theta}
    \begin{pmatrix} 1 & \cos\theta \\
    \cos\theta & 1
    \end{pmatrix}.
\]
See Figure 4.10 in \cite{MR1913879}.
We denote the Markoff map $\psi_{\rho_\theta}$ by $\psi$.

\begin{lemma}\label{lem:mintheta}
We have $\Omega_\psi(2/\sin\theta) = \{ X_{0/1}, X_{1/0} \}$.
\end{lemma}

\begin{proof}
Note that
\begin{align*}
    \psi(X_{-1/1}) & = \frac{2}{\sin^2\theta}, \;
    \psi(X_{ 0/1}) = 
    \psi(X_{ 1/0}) = \frac{2}{\sin  \theta}, \;
    \psi(X_{ 1/1}) = \frac{2}{\sin^2\theta}. \\
    \psi(X_{(n+1)/1}) & = \frac{2}{\sin\theta} \cdot 
    \psi(X_{n/1}) - \psi(X_{(n-1)/1})
\end{align*}
Then, the rest of the proof is the same 
as for Lemma \ref{lem:min}.
\end{proof}

Hence, $|\tr^2\rho_\theta(C) - 4| \geq 
|\tr^2\rho_\theta(A) - 4|$
for any primitive $C\in F_2$.
Since $\rho_\theta(F_2)$ is isomorphic to $F_2$,
by Lemma \ref{lem:commutator}:
\begin{proposition}
For any $\theta$-Schottky group (representation) $\rho_\theta$,
we have
\[
    J(\rho_\theta(F_2)) 
    = J(\rho_\theta(A), \rho_\theta(B))
    = \frac{4\cos^2\theta}{\sin^4\theta}.
\]
In particular,
we have $J(\rho_{\pi/4}(F_2)) = 8$, and
$\lim_{\theta\to 0} J(\rho_\theta(F_2)) = \infty$,
For any real $r$ with $r \geq 8$, 
the J\o rgensen number is realized
by a $\theta$-Schottky group.
\end{proposition}

\begin{remark}
$\rho_{\pi/4}$ ($\theta = \pi/4$) in this subsection and 
$\rho_1$ ($k=1$) in the last subsection are the same representation.
\end{remark}

\section{The diagonal slice of Schottky space}

J\o rgensen numbers can be studied by computer experiments.
In this section, we present a computer generated picture
(Figure \ref{fig:ds}) which estimates J\o rgensen numbers
in the diagonal slice of Schottky space.

\subsection{The diagonal slice of Schottky space}

Let us go back to the singular solid torus in section 2.
Recall that
\begin{equation}\label{eqn:AB}
A_a  = \begin{pmatrix} -3a/2 & 1/2\\ -1/2 & -1/2a \end{pmatrix},
\quad 
B = \begin{pmatrix} i & 0 \\ 0 & -i \end{pmatrix}.
\end{equation}
In this section, we consider that $A_a, B\in\SL(2, \Cset)$.
We denote $\langle A_a, B\rangle$ by $G_a$.
Set 
\begin{equation}\label{eqn:atox}
    x = -\tr A_a^2 
    = - \frac{9 a^2}{4} - \frac{1}{4 a^2} + \frac{1}{2}
\end{equation}
Then, by the trace identity (\ref{eqn:trace}), 
we have 
\begin{equation}\label{eqn:one}    
    \tr A_a   = \sqrt{-x+2},\quad
    \tr B     = 0,\quad
    \tr A_a B = \sqrt{x+1},\quad
    \tr [A_a, B]  = 1. 
\end{equation}
See section 5.0.2
(in particular, Remark 5.4 for the sign of the square roots)
in \cite{STY}.
From now on, we consider that $x$ is a complex parameter.
The locus $\mathcal D$ in the $x$-plane 
of discrete and faithful representations
was fully determined
by computing Keen-Series {\it pleating rays} \cite{STY}.
(Here, ``faithful'' means that $\langle A_a, B\rangle$ is isomorphic to 
$\langle A_1, B\rangle$ as an abstract group.)
$\mathcal D$ is called {\it the diagonal slice of Schottky space}.
See Figure \ref{fig:ds}.
\begin{figure}[t]
\centering
\includegraphics[height=7cm]{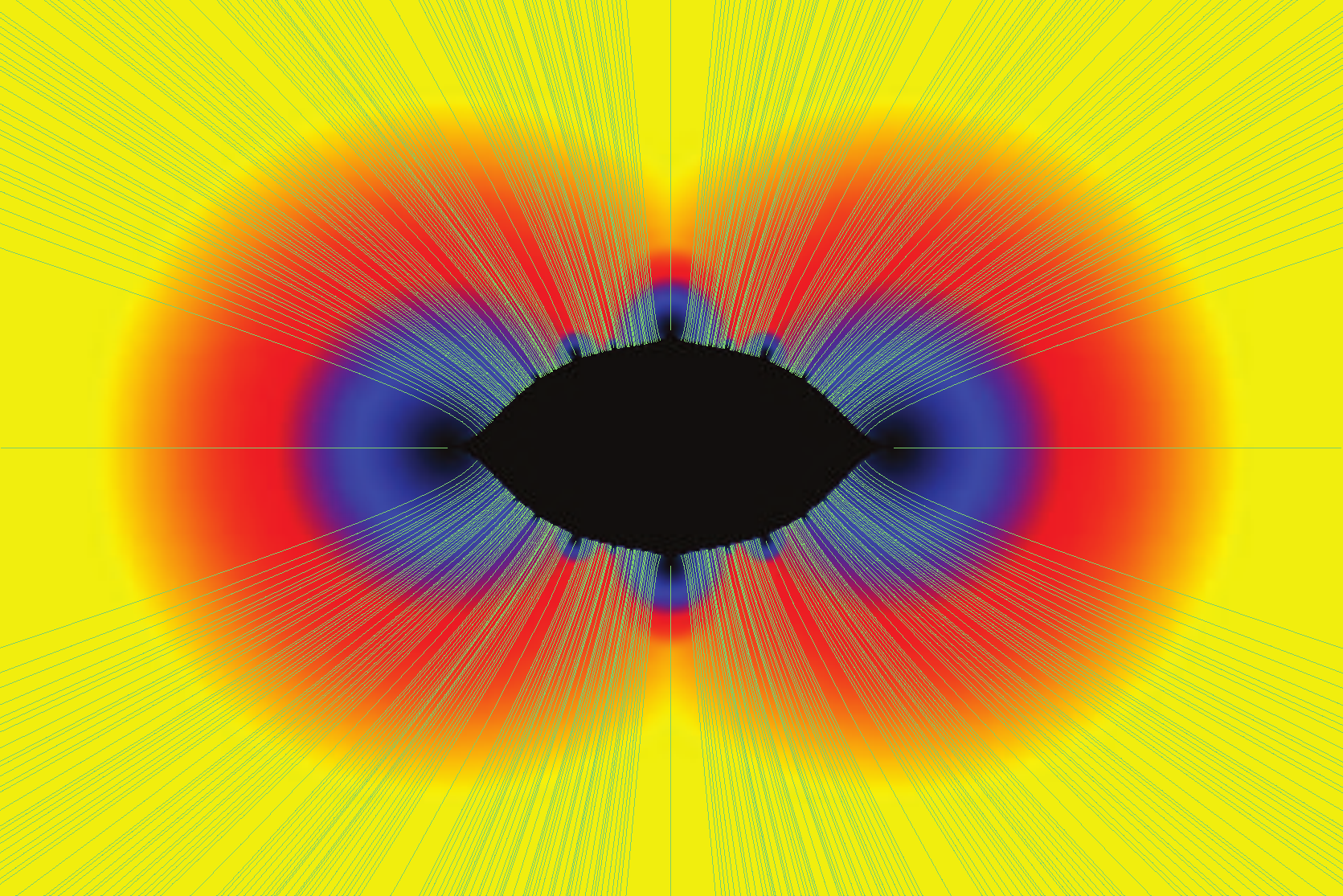}
\hspace{1cm}
\includegraphics[height=6cm]{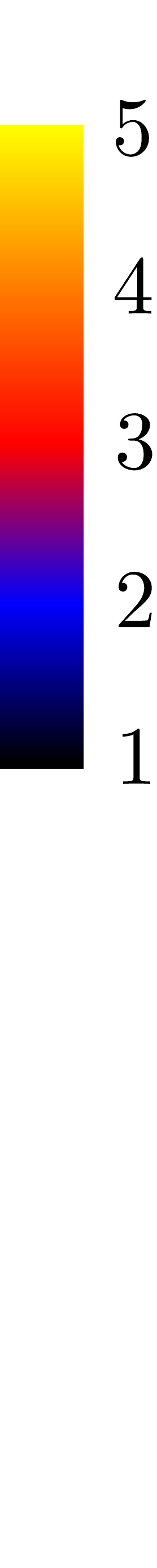}
\caption{The diagonal slice and J\o rgensen number:
$-7 \leq \Re x \leq 8, -5\leq \Im x \leq 5$.
The color indicates estimates of J\o rgensen number from above.}
\label{fig:ds}
\end{figure}
The outside of the {\it center black eye} 
corresponds to $\mathcal D$.
$\mathcal D$ is foliated by pleating rays.
We will describe pleating rays briefly in the next subsection.

Let $a$ be a complex number such that
$x = -\tr A_a^2 \in\mathcal D$.
We denote the representation from
$F_2 = \langle {\mathcal A}, {\mathcal B}\rangle$ to $\SL(2, \Cset)$
which sends $\mathcal A$ to $A_a$ and $\mathcal B$ to $B$ by $\rho_x$.
Let $\psi_x : \widehat\Qset \to \Cset$ be the Markoff map 
associated with $\langle A_a, B\rangle$:
\[
    \psi_x(0/1) = \tr A_a, \; 
    \psi_x(1/0) = \tr B, \; 
    \psi_x(1/1) = \tr A_a B.
\]
If $W$ and $V$ are associated primitives of $F_2$,
then their images in $G_a$ generate $G_a$.
By Lemma \ref{lem:commutator} and
equation (\ref{eqn:one}),
$\tr\rho_x ([W, V]) = \tr [A_a,B] = 1$.
Hence,
\[
    J(\rho_x(W), \rho_x(V)) = |\psi_x^2(n/m) - 4| + 1 \geq J(G_a),
\]
where $n/m$ is determined by the abelianization
$\alpha(W) = {\mathcal A}^m {\mathcal B}^n$.
Hence, 
\[
    \Psi(x) := \inf_{q\in \widehat\Qset} |\psi_x^2(q) - 4| + 1
\]
gives an estimate of $J(G_a)$ from above.
We calculate $\Psi(x)$ by computer
for $\{ x\in\Cset \:|\: -7 \leq \Re(x) \leq 8, -5 \leq \Im(x) \leq 5\}$.
The color 
outside the black eye in Figure \ref{fig:ds} indicates 
the values of $\Psi(x)$.
Since thees groups are discrete and faithful (in particular non-elementary), 
$\Psi(x) \geq 1$.
Since $|\psi^2_x(1/0) - 4| + 1 = 5$,
we have $\Psi(x) \leq 5$.
Figure \ref{fig:ds} suggests that
for each $1 \leq r \leq 5$,
there are many non-elementary Kleinian groups 
in the diagonal slice with J\o rgensen number $r$.

\begin{remark}
In practice, we can not calculate $\psi_x^2(q)$ 
for all $q\in \widehat \Qset$,
and Figure \ref{fig:ds} is an approximation of $\Psi(x)$
by calculating $\psi_x^2(q)$ for many $q \in \widehat\Qset$.
But, if $x$ is in the {\it Bowditch set},
in principle, 
we can calculate $\Psi(x)$.
The key is Lemma 3.24 in \cite{MR2370281}.
The diagonal slice of Schottky space seems to
coincide with the Bowditch set.
See Section 2, \cite{STY}.
\end{remark}

\begin{question}
Does the equation $J(G_a) = \Psi(x)$ hold
for each $x\in\mathcal D$?
Note that, since $G_a$ is not free, 
there might be generating pairs of $G_a$
which do not come from associated primitives in $F_2$.
\end{question}

\subsection{J\o rgensen groups on the boundary of the diagonal slice}

We consider the case $\Psi(x) = 1$ on $\partial \mathcal D$.
We begin by describing pleating rays very briefly.
See \cite{STY} for detail.

For $p/q\in\widehat\Qset$, 
the {\it real trace locus} $\Rset_{p/q}$ of $p/q$ is 
$\{ x\in\Cset \:|\: \psi_x(p/q) \in (-\infty, -2] \cup [2, \infty) \}$.
The rational pleating ray ${\mathcal P}_{p/q}$ is 
a union of connected non-singular
branches of $\Rset_{p/q}$ (Corollary 4.11 \cite{STY}).
The rational pleating rays are indexed by $\widehat\Qset/\sim$,
where $p/q\sim p'/q'$ if and only if $p'/q' = \pm p/q + 2k, k\in \Zset$
(Proposition 4.8, \cite{STY}).
(If $p/q\sim p'/q'$, then we have $\psi_x(p/q) = \pm \psi_x(p'/q')$.)
For example,
the condition $a\in\Rset$, $1 \leq a \leq (\sqrt{7}+2)/3$ 
in subsection \ref{subsec:JN}
corresponds to
$\{x\in\Rset\:|\:-5 \leq x \leq -2\} \subset {\mathcal P}_{0/1}$.

If $p/q$ is not an integer,
${\mathcal P}_{p/q}$ consists of two components (branches).
One is in the upper-half plane, and the other in the lower-half plane.
Let $e_{p/q}$ be the end point of ${\mathcal P}_{p/q}$
in the upper-half plane,
or on the real axis if $p/q$ is an integer.
By construction, we have $\psi_{e_{p/q}}^2(p/q) - 4 = 0$.
Hence, $\Psi(e_{p/q}) = 1$,
and the corresponding representation is a J\o rgensen group of parabolic type.

In order to describe these groups,
we recall the Li-Oichi-Sato normalization:
\begin{lemma}[Lemma 3.1 \cite{MR2205423}]
Let $M$ and $N$ be elements of $\PSL(2, \Cset)$
such that $M$ is parabolic and $N$ is elliptic or loxodromic.
Then, $M, N$ can be normalized as
\[
    M = \begin{pmatrix} 1 & 1 \\ 0 & 1 \end{pmatrix}, \quad
    N_{\sigma, \mu} = \begin{pmatrix} \mu\sigma & \mu^2\sigma - 1/\sigma \\
                              \sigma & \mu\sigma  \end{pmatrix},
\]
where $\sigma\in\Cset\setminus\{0\}$ and $\mu\in\Cset$.
\end{lemma}

\begin{remark}\label{rem:LOS}
Li-Oichi-Sato 
\cite{MR2205423} \cite{MR2107659} \cite{MR2153913}
considered the case $\mu = i k$, $k\in\Rset$.
They conjectured that,
for any J\o rgensen group $G$ of parabolic type,
there exists a marked group 
$G_{\sigma, ik} = \langle M, N_{\sigma,ik}\rangle$,
$\sigma \in\Cset\setminus\{0\}$,
$k\in\Rset$, such that $G_{\sigma,ik}$ is conjugate to $G$.
Later, Callahan \cite{MR2525101} found counter
examples for this conjecture.
\end{remark}

Now, we compute some examples.
We begin by calculating some Markoff maps 
and the end point of pleating rays:
\begin{align*}
    \psi_x^2(0/1)-4 &= (-x+2)-4,        & e_{0/1} &= -2\\
    \psi_x^2(1/3)-4 &= (-x+1)^2(x+1)-4, & e_{1/3} &\approx -0.5652 + 1.0434 i,\\
    \psi_x^2(3/8)-4 &= (-x+2)(x+1)(x^3-x^2-1)^2-4, &e_{3/8} &\approx -0.2992 + 1.0726 i,\\
    \psi_x^2(2/5)-4 &= x^4(-x+2)-4,     & e_{2/5} &\approx -0.1372 + 1.1260 i,\\
    \psi_x^2(1/2)-4 &= (-x+2)(x+1)-4,   & e_{1/2} &= (1+\sqrt{7}i)/2.
\end{align*}
(Note that the equation $\psi_x^2(p/q) = 4$ have many roots, 
and in order to make the right choice for $e_{p/q}$, 
we need \cite{STY}.)
Let $A_{p/q}$ denote the matrix in (\ref{eqn:AB}) 
such that $e_{p/q} = - \tr A_a^2$.
Then, for example, 
\begin{align*}
    &(A_{0/1}, B), \:
    (A_{1/3}^3 B, A_{1/3}), \:
    (A_{3/8}^3 B A_{3/8}^3 B A_{3/8}^2 B, A_{3/8}^3 B), \\
    &(A_{2/5}^3 B A_{2/5}^2 B, A_{2/5}^3 B), \:
    (A_{1/2}^2 B, A_{1/2})
\end{align*}
generate J\o rgensen groups of parabolic type.
The Li-Oichi-Sato parameters for these groups are as follows:
\begin{align*}
    (\sigma, \mu) & =       ( i, 0                ) & \text{for } & (A_{0/1}, B) \\
    (\sigma, \mu) & \approx (-i, 0.1597 + 0.8166 i) & \text{for } & (A_{1/3}^2 B, A_{1/3}) \\
    (\sigma, \mu) & \approx ( i, 0.1839 + 0.9356 i) & \text{for } & (A_{3/8}^3 B A_{3/8}^3 B A_{3/8}^2 B, A_{3/8}^3 B) \\
    (\sigma, \mu) & \approx ( i, 0.3016 + 0.9041 i) & \text{for } & (A_{2/5}^3 B A_{2/5}^2 B, A_{2/5}^3 B) \\
    (\sigma, \mu) & =       ( i, (1 + \sqrt{7}i)/4) & \text{for } & (A_{1/2}^2 B, A_{1/2})
\end{align*}
We conjecture that 
groups which correspond to
the end points of the rational pleating rays
(except $-2$ and $3$ on the real axis)
are counter examples for Li-Oichi-Sato's conjecture
in Remark \ref{rem:LOS}.

\begin{remark}
Let $y$ be an element of $\mathcal D$ such that
$\psi_y$ has a sequence $\{q_n\}_{n\in\Zset}$
of distinct elements of $\widehat\Qset/\sim$
with $\lim_{n\to\infty}\psi_y^2(q_n) = 4$.
Then $\Psi(y) = 1$ and $y$ corresponds to a 
geometrically infinite group
with unbounded geometry.
\end{remark}

\begin{remark}
The limit set of 
$\langle M, N_{\sigma, \mu}\rangle$ with
$(\sigma, \mu) = (i, 0)$ is the real line.
The limit sets of other examples are complicated and beautiful.
For example, 
see Figure~\ref{fig:limitset2} for
the limit set of the third case 
$(\sigma, \mu) \approx (i, 0.1839 + 0.9356 i)$.
\begin{figure}[ht]
\centering
\includegraphics[height=7cm]{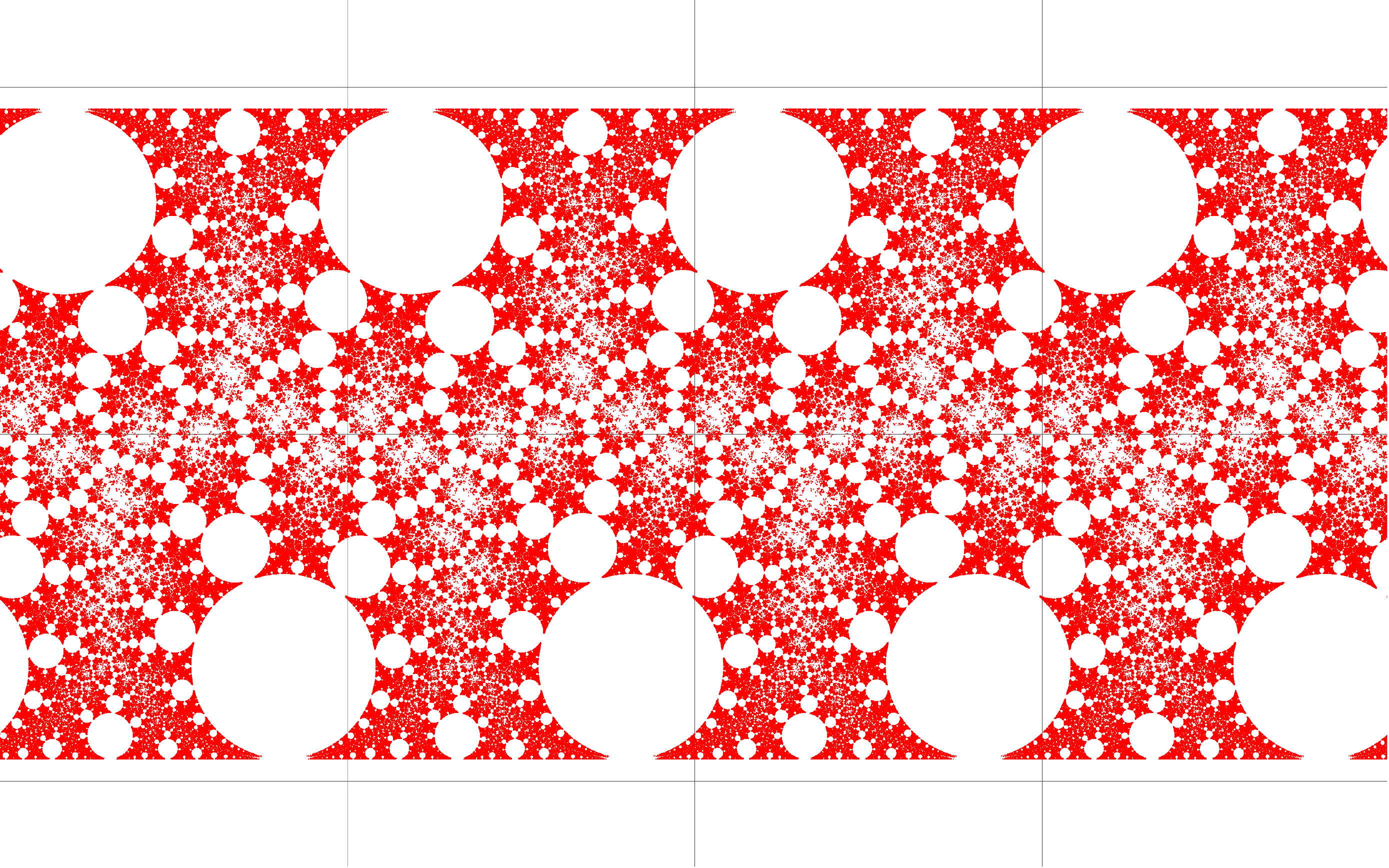}
\caption{The limit set of $\langle M, N_{\sigma, \mu}\rangle$ 
with $(\sigma, \mu) \approx (i, 0.1839 + 0.9356 i) $} 
\label{fig:limitset2}
\end{figure}
\end{remark}


\end{document}